\documentclass{article}
\usepackage{amsfonts}
\usepackage{amsfonts}
\usepackage{amsfonts}
\usepackage{amsfonts}
\usepackage{amsfonts}
\usepackage{amsfonts}
\usepackage{amsfonts}
\usepackage{mathrsfs}
\usepackage{CJK,amssymb,amsmath,amsfonts,amsthm,graphicx,enumerate,amscd}
\usepackage{bbm}
\usepackage{amssymb}
\usepackage{graphicx,verbatim}
\usepackage{amscd}
\usepackage{amsmath}
\usepackage{amsfonts}
\usepackage{graphicx}
\usepackage{latexsym}

\newtheorem{thm}{Theorem}[section]
\newtheorem{lem}{Lemma}[section]

\newtheorem{prop}{Proposition}[section]
\newtheorem{defn}{Definition}[section]
\newtheorem{rmk}{Remark}[section]
\newtheorem{eg}{Example}[section]

\begin{document}
\begin{center}
\Large{ \bf Embeddings from noncompact
symmetric spaces to their compact duals}
\end{center}

\begin{center}
Yunxia Chen \& Yongdong Huang \& Naichung Conan Leung
\end{center}

\begin{abstract}
Every compact symmetric space $M$ admits a dual noncompact symmetric space $\check{M}$. When $M$ is a generalized Grassmannian, we can view $\check{M}$ as a open submanifold of it consisting of space-like subspaces \cite{HL}. Motivated from this, we study the embeddings from noncompact symmetric spaces to their compact duals, including
space-like embedding for generalized Grassmannians, Borel embedding for Hermitian symmetric spaces and the generalized embedding for symmetric R-spaces. We will compare these
embeddings and describe their images using cut loci.
\end{abstract}

\bigskip
\section{Introduction}
Riemannian symmetric spaces of semi-simple type come in pairs, the
compact ones and their noncompact duals. The duality between the
compact types and noncompact types not
only provides two viewpoints of the classification problem of
symmetric spaces, but also explains the formal analogy between
spherical trigonometry and hyperbolic trigonometry.

In a recent paper
 \cite{HL} by Y.D. Huang and N.C. Leung, they give a uniform description of classcial
 symmetric spaces (up to finite covers and abelian parts): compact classical symmetric spaces
 described
 as some kinds of Grassmannian and their noncompact duals as the Grassmannians of space-like
 subspaces. That means, for these generalized
Grassmannians, we have the \textbf{space-like embedding} \textbf{\textit{p}} from noncompact
types to their compact duals.

It is also well-known that, for Hermitian symmetric spaces, we have the \textbf{Borel embedding} \textbf{\textit{b}} from noncompact
ones to their compact duals.

Are these two embeddings \textbf{\textit{p}} and \textbf{\textit{b}} compatible with each other? Can they be generalized to all the  symmetric spaces?
How to characterize their images? These are the questions we want to solve in this paper.

One of the key point to solve the first two questions is the fact that both the generalized Grassmannians and the Hermitian symmetric spaces have
the following property: the compact symmetric space admits a connected Lie transformation group larger than the isometry group (definition of symmetric R-space), and hence can be
described as a quotient of a semisimple Lie group by a parabolic subgroup \cite{Na}. Using this characterization, for all the symmetric R-spaces (which include all the compact generalized Grassmannian and compact Hermitian symmetric spaces),
we can construct an open embedding \textbf{\textit{g}} from its noncompact dual to it (section 3.3),
which coincides  with both \textbf{\textit{p}} and \textbf{\textit{b}}.

One of the key point to solve the third question is the fact that the \textbf{unit lattice} of a symmetric R-space is orthonormal
(another characterization of symmetric R-spaces) \cite{Lo}. With this property,
it is easy to compute the cut loci of this symmetric R-space, which ensures us
to construct an explicit embedding \textbf{\textit{f}}
from its noncompact dual to it (section 4.3). From the definition of \textbf{\textit{f}}, we can easily
see that its image is just ``half of the interior points'' of the compact symmetric spaces.

Comparing the embeddings \textbf{\textit{f}} and \textbf{\textit{g}}, roughly speaking,
they both coincide with the space-like embedding in the generalized
Grassmannian case  and the Borel embedding in the Hermitian case,
are they equal to each other? The answer is affirmative by using such
characterization of a symmetric R-space that it is a compact Hermitian symmetric space or a real form
of a compact Hermitian symmetric space.

The organization of this paper is as follows. Section 2 is a brief review of the duality between symmetric spaces. In section 3, we first study the space-like
 embedding \textbf{\textit{p}} for generalized Grassmannians, then review the Borel embedding \textbf{\textit{b}} for Hermitian symmetric spaces, and finally
 we study the embedding \textbf{\textit{g}} (which generalizes both \textbf{\textit{p}} and \textbf{\textit{b}}) for symmetric R-spaces.
In section 4, we construct the embedding \textbf{\textit{f}} using cut loci. At last, in section 5,
we compare these embeddings and show that \textbf{\textit{g}}=\textbf{\textit{f}} for all the symmetric R-spaces except $\text{SO}(n+2)/
\text{SO}(n)\text{SO}(2)$ and $\text{SO}(n+1)\times \text{SO}(m+1)/
S(\text{O}(n) \times \text{O}(m))$.

\bigskip

$\mathbf{Acknowledgements.}$ We are grateful to
J.J. Zhang for many useful comments and discussions.
The first author is supported by the National Natural
Science Foundation of China (No. 11501201). The work of Leung described in this paper was substantially supported by a grant from
the Research Grants Council of the Hong Kong Special Administrative Region, China (Project
No. CUHK14302015) and a direct grant from The Chinese University of Hong Kong (Project No. 4053161).

\bigskip

\section{The duality for symmetric spaces}
Let $M=G/K$ be an irreducible Riemannian symmetric space and
$\mathfrak{g}=\mathfrak{k}\bigoplus\mathfrak{m}$ be the corresponding Cartan
decomposition. For the Lie algebra $\mathfrak{g}$, we have the Killing
form $B$, defined by $B(X,Y):=Tr(adX\cdot adY)$, for $X$, $Y$ in
$\mathfrak{g}$. Then $\mathfrak{k}$ and $\mathfrak{m}$ are
orthogonal with respect to the killing form $B$, and
$B|_{\mathfrak{k}}$ is negative definite, i.e.
$B|_{\mathfrak{k}}<0$. The symmetric space $M=G/K$ is said to be of \textbf{compact} type,
\textbf{noncompact} type, Euclidean type if $B|_{\mathfrak{m}}>0$,
$B|_{\mathfrak{m}}<0$, $B|_{\mathfrak{m}}=0$ respectively. Both the
compact type and noncompact type are called semi-simple type.

There is a remarkable and important duality between Riemannian symmetric spaces of
compact type and noncompact type. Given a noncompact symmetric space
$M=G/K$ with Cartan decomposition $\mathfrak{g}=\mathfrak{k}\bigoplus\mathfrak{m}$. Define
$\check{\mathfrak{g}}=\mathfrak{k}\bigoplus \sqrt{-1}\mathfrak{m}$,
then the Lie algebra pair $(\check{\mathfrak{g}},\mathfrak{k})$ gives
us a finite number of compact symmetric spaces which have the same universal
covering space. We call these compact symmetric spaces the compact
duals of $M$. Conversely, given a compact symmetric space, we can
get its noncompact dual in the similar way, the noncompact dual is unique
since all the noncompact symmetric spaces are simply-connected.

The duality between the
compact type  and noncompact type  not
only provides two viewpoints of the classification problem of
symmetric spaces, but also explains the formal analogy between
spherical trigonometry and hyperbolic trigonometry.

\begin{eg} (Page 242 in \cite{H})
The sphere $S^2=SO(3)/SO(2)$ is dual to the two dimensional Lobachevsky space $H^2=SO_0(2,1)/SO(2)$.

For the sphere, the formulas
\[
\frac{\sin{a}}{\sin{A}}=\frac{\sin{b}}{\sin{B}}=\frac{\sin{c}}{\sin{C}}
\]
\[
\cos{a}=\cos{b}\cos{c}+\sin{b}\sin{c}\cos{A}
\]
hold for a geodesic triangle with angles A, B, C and sides of length a, b, c.

For the two dimensional Lobachevsky space, the corresponding formulas are
\[
\frac{\sinh{a}}{\sin{A}}=\frac{\sinh{b}}{\sin{B}}=\frac{\sinh{c}}{\sin{C}}
\]
\[
\cosh{a}=\cosh{b}\cosh{c}+\sinh{b}\sinh{c}\cos{A}
\]

Since $\sinh{iz}=i\sin{Z}$ and $\cosh{iz}=\cos{z}$, the two sets of formulas correspond to each other under the substitution
$a\rightarrow ia$, $b\rightarrow ib$, $c\rightarrow ic$.

\end{eg}

Besides that, we can use the duality to study many aspects of Riemannian symmetric spaces such as polar actions on Riemannian symmetric spaces \cite{Ko},
Fourier transforms on Riemannian symmetric spaces \cite{B} and so on.

\bigskip

\section{The embeddings from Lie group aspect}
In this section, we will first study the space-like embedding \textbf{\textit{p}} for generalized Grassmannians,
 then review the Borel embedding \textbf{\textit{b}} for Hermitian symmetric spaces, and finally show that the embedding \textbf{\textit{g}}
for symmetric R-spaces is a generalization of these two embeddings \textbf{\textit{p}} and \textbf{\textit{b}}.

\subsection{Space-like embedding for generalized Grassmannians}
In this subsection, we will first review the
Grassmannian description of classical symmetric spaces given by Y.D. Huang and N.C. Leung in \cite{HL} and then
study the corresponding space-like embedding.

For the classical compact
symmetric spaces (Here classical means the corresponding Lie groups are matrix groups),
we have a uniform description of them as generalizations of Real Grassmannian
$\text{O}(n+m)/
\text{O}(n)\text{O}(m)=\{L\subset\mathbb{R}^{n+m}: L\cong \mathbb{R}^n\}$ using two
(associative) normed division algebras.
That is, the four types of compact generalized Grassmannians: (i) Grassmannian, (ii)
Lagrangian Grassmannian, (iii) double Lagrangian Grassmannian, (iv)
maximum isotropic Grassmannian give us all possible types of irreducible
classical compact symmetric spaces up to a finite cover and some
abelian part \cite{HL}. We list all these compact symmetric spaces in the following:

First type: \emph{Grassmannians}
\[
Gr_{\mathbb{AB}}\left(  k,n\right)  =\left\{  \left(
\mathbb{A\otimes B}\right)  ^{k}\subset\left(  \mathbb{A\otimes
B}\right)  ^{n}\right\}  .
\]

\begin{center}
\renewcommand{\arraystretch}{2}
\begin{tabular}
[c]{|c||c|c|c|}\hline $\mathbb{A}\backslash\mathbb{B}$ &
$\mathbb{R}$ & $\mathbb{C}$ & $\mathbb{H}$
\\\hline\hline $\mathbb{R}$ &
$\displaystyle\frac{\text{O}(n)}{\text{O}(k)\text{O}(n-k)}$ &
$\displaystyle\frac{\text{U}(n)}{\text{U}(k)\text{U}(n-k)}$ &
$\displaystyle\frac{\text{Sp}(n)}{\text{Sp}(k)\text{Sp}(n-k)}$
\\\hline $\mathbb{C}$ &
$\displaystyle\frac{\text{U}(n)}{\text{U}(k)\text{U}(n-k)}$ &
$\displaystyle\frac{\text{U}(n)^{2}}{\text{U}(k)^{2}\text{U}(n-k)^{2}}$&
$\displaystyle\frac{\text{U}(2n)}{\text{U}(2k)\text{U}(2n-2k)}$
\\\hline $\mathbb{H}$ &
$\displaystyle\frac{\text{Sp}(n)}{\text{Sp}(k)\text{Sp}(n-k)}$ &
$\displaystyle\frac{\text{U}(2n)}{\text{U}(2k)\text{U}(2n-2k)}$ &
$\displaystyle\frac{\text{O}(4n)}{\text{O}(4k)\text{O}(4n-4k)}$
\\\hline
\end{tabular}
\newline(Table: C1)
\end{center}

\bigskip

Second type: \emph{Lagrangian Grassmannians}
\[
LGr_{\mathbb{AB}}\left(  n\right)  =\left\{  \left(  \frac{\mathbb{A}}%
{2}\mathbb{\otimes B}\right)  ^{n}\subset\left(  \mathbb{A\otimes
B}\right) ^{n}\right\}  ,
\]
where $\frac{\mathbb{A}}{2}$ denotes $\mathbb{R},\mathbb{C}$ when
$\mathbb{A}$ is $\mathbb{C},\mathbb{H}$ respectively.

\begin{center}
\renewcommand{\arraystretch}{2}
\begin{tabular}
[c]{|c||c|c|c|}\hline $\mathbb{A}\backslash\mathbb{B}$ &
$\mathbb{R}$ & $\mathbb{C}$ & $\mathbb{H}$
\\\hline\hline $\mathbb{C}$ &
$\displaystyle\frac{\text{U}(n)}{\text{O}(n)}$ &
$\displaystyle\frac{\text{U}(n)^{2}}{\text{U}(n)}$ &
$\displaystyle\frac {\text{U}(2n)}{\text{Sp}(n)}$
\\\hline $\mathbb{H}$ &
$\displaystyle\frac{\text{Sp}(n)}{\text{U}(n)}$ &
$\displaystyle\frac{\text{U}(2n)}{(\text{U}(n)^{2})}$ &
$\displaystyle\frac {\text{SO}(4n)}{\text{U}(2n)}$ \\\hline
\end{tabular}
\newline(Table: C2)
\end{center}

\bigskip

Third type: \emph{Double Lagrangian Grassmannians}
\[
LLGr_{\mathbb{AB}}(n)=\left\{  \left(\frac{\mathbb{A}}{2}\otimes\frac{\mathbb{B}}%
{2}\right)^{n}\oplus \left(j_{1}\frac{\mathbb{A}}{2}\otimes j_{2}\frac{\mathbb{B}}{2}%
\right)^{n}\subset\left(  \mathbb{A\otimes B}\right)  ^{n}\right\} ,
\]
where $j_{1}\in\mathbb{A}$ is $i$, $j$  when $\mathbb{A}$ is
$\mathbb{C}$, $\mathbb{H}$  respectively, and $j_{2}\in\mathbb{B}$
is similar.

\begin{center}
\renewcommand{\arraystretch}{2}
\begin{tabular}
[c]{|c||c|c|}\hline $\mathbb{A}\backslash\mathbb{B}$ & $\mathbb{C}$
& $\mathbb{H}$
\\\hline\hline $\mathbb{C}$ &
$\displaystyle\frac{\text{U}(n)^{2}}{\text{O}(n)^{2}}$ &
$\displaystyle\frac{\text{U}(2n)}{\text{O}(2n)}$
\\\hline $\mathbb{H}$ &
$\displaystyle\frac{\text{U}(2n)}{\text{O}(2n)}$ &
$\displaystyle\frac{\text{SO}(4n)}{\text{S(O}(2n)^{2})}$
\\\hline
\end{tabular}
\newline(Table: C3)
\end{center}

Fourth type: \emph{Compact (semi-)simple Lie groups}
\[
G_{\mathbb{AB}}\left(  n\right)  =\left\{  \left(  \mathbb{A\otimes
B}\right) ^{n}\subset\left(  \mathbb{A\otimes B}\right)
^{n}\oplus\left( \mathbb{A\otimes B}\right)  ^{n}\right\}
^{\sigma_{g^{\prime}}},
\]
here $\sigma_{g^{\prime}}$ is an involution induced by a canonical
symmetric
$2$-tensor on $(\mathbb{A}\otimes\mathbb{B})^{n}\oplus(\mathbb{A}%
\otimes\mathbb{B})^{n}$.
\begin{center}
\renewcommand{\arraystretch}{1.3}
\begin{tabular}
[c]{|c||c|c|c|}\hline $\mathbb{A}\backslash\mathbb{B}$ &
$\mathbb{R}$ & $\mathbb{C}$ & $\mathbb{H}$
\\\hline\hline $\mathbb{R}$ &
$\text{SO}(n)$ & $\text{U}(n)$ & $\text{Sp}(n)$
\\\hline $\mathbb{C}$ &
$\text{U}(n)$ & $\text{U}(n)^{2}$ & $\text{U}(2n)$
\\\hline $\mathbb{H}$ &
$\text{Sp}(n)$ & $\text{U}(2n)$ & $\text{SO}(4n)$ \\\hline
\end{tabular}
\newline(Table: C4)
\end{center}

This uniform description give us many new insights to the intimate relationships
among different symmetric spaces and to the geometries of special holonomy \cite{LNC1}\cite {LNC2}\cite{LNC3}.

\begin{rmk}
From the above tables $(C1)-(C4)$, we obtain all the classical
compact symmetric spaces up to some abelian part and a finite cover.
Note that in the rest of this paper,
when we say \textbf{generalized Grassmannian}, we only mean those symmetric
spaces in these tables and their noncompact duals.
\end{rmk}

In the generalized Grassmannian case, the noncompact
symmetric spaces can be described as space-like Grassmannians (see the arguments below or \cite{HL}). That
means they are open submanifolds of their compact duals consisting
of space-like linear subspaces. Hence
we can always embed the noncompact generalized Grassmannians to their compact duals.
We will use \textbf{\textit{p}} to denote this space-like embedding in
the generalized Grassmannian case.

\begin{eg}
Real Grassmannian manifold $M^c=\text{O}(n+m)/ \text{O}(n)\text{O}(m)=\{L\subset\mathbb{R}^{n+m}: L\cong \mathbb{R}^n\}$
$(m\geqslant n)$.

The noncompact dual of $M^c$
is given by $M^n=\text{O}(n,m)/
\text{O}(n)\text{O}(m)$.

Define a metric $g$ on $\mathbb{R}^{n+m}$ to be $g((u_1,\cdots,
u_{n+m})^t,(v_1,\cdots, v_{n+m})^t)=u_1v_1+\cdots
+u_nv_n-u_{n+1}v_{n+1}\cdots -u_{n+m}v_{n+m}$, denote $M^+:=\{L \in
M^c: g|_L > 0\}$, such $L \in
M^c$ with the property that $g|_L > 0$ is called a space-like subspace.

\bigskip
\textbf{Fact 1}: $M^+ \cong \text{O}(n,m)/ \text{O}(n)\text{O}(m)=M^n$.
\begin{proof}
First we write the subspace $L$ in the form of matrix: Choose basis of $L$, namely
$\{l_1, \cdots ,l_n\}$, then we have $L=[l_1, \cdots, l_n]$, a
$(n+m)\times n$ matrix.

The action of $\text{O}(n,m)$ on $M^+$ is given by $A \cdot L=AL$ as
matrix product.

If $X \in L \in M^+$, i.e. $X^t \left(
\begin{array}{cc}
I & 0\\
0 & -I \end{array} \right) X> 0$, then $(AX)^t \left(
\begin{array}{cc}
I & 0\\
0 & -I \end{array} \right) (AX)=X^tA^t \left(
\begin{array}{cc}
I & 0\\
0 & -I \end{array} \right) AX=X^t \left(
\begin{array}{cc}
I & 0\\
0 & -I \end{array} \right) X> 0$, for any $ A \in \text{O}(n,m)$.

Hence $A \cdot L \in M^+$, for any $ A \in \text{O}(n,m)$,
$L \in M^+$.

Transitive: Write $L_o=[e_1, \cdots, e_n]= \left[
\begin{array}{cc}
I_{n\times n} \\
0_{m \times n}  \end{array} \right]$. For any $ L \in M^+$, we
can choose a basis of $L$, such that $L=\left[
\begin{array}{cc}
I_{n\times n} \\
Y_{m \times n}  \end{array} \right]$. Since $L \in M^+$, the matrix
$I-Y^tY$ is positive definite, in particular, it is invertible. Now
want to find $A \in \text{O}(n,m)$ such that $A \cdot L_o=\left(
\begin{array}{cc}
A_1 & A_2\\
A_3 & A_4 \end{array} \right)\left[
\begin{array}{cc}
I \\
0  \end{array} \right]=\left[
\begin{array}{cc}
A_1\\
A_3 \end{array} \right]=\left[
\begin{array}{cc}
I \\
Y \end{array} \right]$. We only need to take $A=\left(
\begin{array}{cc}
(I-Y^tY)^{-\frac{1}{2}} & Y^t(I-YY^t)^{-\frac{1}{2}}\\
Y(I-Y^tY)^{-\frac{1}{2}} & (I-YY^t)^{-\frac{1}{2}} \end{array}
\right)$, note the matrices here can take square roots since they are positive definite symmetric matrices.

Isotropic subgroup: By $A \cdot L_o=\left(
\begin{array}{cc}
A_1 & A_2\\
A_3 & A_4 \end{array} \right)\left[
\begin{array}{cc}
I \\
0  \end{array} \right]=\left[
\begin{array}{cc}
A_1\\
A_3 \end{array} \right]=\left[
\begin{array}{cc}
I \\
0 \end{array} \right]$, we have $A_3=0,A_2=0,A_1 \in \text{O}(n),
A_4 \in \text{O}(m)$. Hence the isotropic subgroup is isomorphic to
$\text{O}(n)\text{O}(m)$.

So we have $M^+ \cong \text{O}(n,m)/ \text{O}(n)\text{O}(m)=M^n$.

\end{proof}

\bigskip

From the above fact, we have the space-like embedding $\textbf{\textit{p}}: M^n \hookrightarrow M^c$ obtained by realized
$M^n$ as the space of all space-like subspaces in $R^{n,m}$.

\textbf{Fact 2}: The embedding \textbf{\textit{p}} is $K$-equivariant.
\begin{proof}
For any $L=\left[
\begin{array}{cc}
I \\
Y \end{array} \right] \in M^n$ and $k=\left(
\begin{array}{cc}
K_1 & 0\\
0 & K_2 \end{array} \right) \in K=O(n)O(m)$,
\[
k \cdot L=\left[
\begin{array}{cc}
I \\
K_2YK_1^{-1} \end{array} \right]
\]
obviously, $\textbf{\textit{p}}(k \cdot L)=k\cdot \textbf{\textit{p}}(L)$.
\end{proof}

\bigskip

Now we describe the embedding $\textbf{\textit{p}}$ from the Lie group theoretical aspect. Since we have
 $M^c\cong GL(n+m, \mathbb{R})/P$, where $P=\{\left(
\begin{array}{cc}
U_{n \times n} & W_{n \times m}\\
0_{m \times n} & V_{m \times m} \end{array} \right) \in GL(n+m, \mathbb{R})\}$,
we can define an embedding $\textbf{\textit{g}}: \text{O}(n,m)/ \text{O}(n)\text{O}(m) \hookrightarrow
GL(n+m, \mathbb{R})/P \cong \text{O}(n+m)/ \text{O}(n)\text{O}(m)$ by $AK\mapsto AP \mapsto ABK$, where
$B \in P$ such that $AB\in \text{O}(n+m)$. We can use the Gram-Schmidt process to find such a $B$.

\textbf{Fact 3}: $\textbf{\textit{g}}=\textbf{\textit{p}}$.
\begin{proof}
Only need to show that for any $A \in \text{O}(n,m)$ , $B \in P$, $AK$ and $ABK$ represent the same subspace,
which is obviously by direct computations.
\end{proof}

\end{eg}

We generalize these properties for real Grassmannians to all generalized Grassmannians.

Recall that, the compact (resp. noncompact) generalized Grassmannians can be viewed as
fixed point sets in the compact (resp. noncompact) Real Grassmannians of some isometric involutions \cite{HL}.
Namely, they are (intersections of) reflective submanifolds
of some real Grassmannians.
Reflective submanifolds of symmetric spaces are studied by S.P. Leung in \cite{L1}\cite{L2}\cite{L3}\cite{L4}.
Since the duality between compact symmetric spaces and noncompact symmetric spaces
is invariant under reflective submanifold and the noncompact real Grassmannian consists of space-like subspaces (Fact 1 of Example 3.1),
we know that if the compact generalized Grassmannian consists of subspaces invariant under some isometric involutions, then
its noncompact dual consists of space-like subspaces which are also invariant under the same isometric involutions (generalization of Fact 1).

Hence we have the space-like embedding \textbf{\textit{p}} defined the same way with the real Grassmannian case. In fact, this space-like
 embedding is just the restriction of the space-like embedding for the corresponding real Grassmannians. From Fact 2 of Example 3.1, we have
the embedding \textbf{\textit{p}} is $K$-equivariant for all the generalized Grassmannians (generalization of Fact 2).

And also for each of these compact generalized Grassmannians $G^c/K$, we can find a connected transformation
group $L$ larger than the isometry group $G^c$ such that $G^c/K \cong L/P$ where
$P$ is a parabolic subgroup of $L$ and $G^c \bigcap P=K$ \cite{KN}\cite{Na}. At this time, for its noncompact dual
$G^n/K$, $G^n$ is also a subgroup of $L$ and $G^n \bigcap P=K$.
The following natural embeddings
\[
G^n \subset L  ~  \textrm{and}  ~   G^c \subset L
\]
induce the natural embeddings
\[
G^n/K \hookrightarrow L/P  ~  \textrm{and}  ~  G^c/K \hookrightarrow L/P.
\]
The latter embedding is in fact an isomorphism, hence we have the following map
\[
\textbf{\textit{g}}: G^n/K \hookrightarrow L/P \cong  G^c/K
\]
\[
AK\mapsto AP\mapsto ABK
\]
where $B \in P$ such that $AB \in G^c$ and such $B$ is unique up to $K$.

Comparing these two embeddings \textbf{\textit{p}} and \textbf{\textit{g}}, they can both be viewed
as the restrictions of \textbf{\textit{p}} and \textbf{\textit{g}} respectively for the corresponding real Grassmannians.
From Fact 3 of Example 3.1, we have
\textbf{\textit{p}}=\textbf{\textit{g}} for all the generalized Grassmannians (generalization of Fact 3).

\bigskip

\subsection{Borel embedding for Hermitian symmetric spaces}
In this subsection, we will review the
Borel embedding for Hermitian symmetric spaces. We refer to
\cite{Mok} for basics about Hermitian symmetric spaces. Hermitian symmetric spaces are
all simply-connected, there is a one-one corresponding between
compact Hermitian symmetric spaces and noncompact Hermitian
symmetric spaces. Let $M^c=G^c/K$ be a compact Hermitian symmetric
space and $M^n=G^n/K$ its noncompact dual. The corresponding Cartan
decompositions are $\mathfrak{g}^c=\mathfrak{k}+\mathfrak{m}_c$ and
$\mathfrak{g}^n=\mathfrak{k}+\mathfrak{m}_n$. Complexify
$\mathfrak{g}^c$ and $\mathfrak{g}^n$, get $\mathfrak{g}^c \otimes
\mathbb{C}=\mathfrak{g}^n \otimes
\mathbb{C}=\mathfrak{g}^{\mathbb{C}}=\mathfrak{k}^{\mathbb{C}}+
\mathfrak{m}^{\mathbb{C}}$, denote by $G^{\mathbb{C}}$ the
corresponding Lie group of $\mathfrak{g}^{\mathbb{C}}$. Decompose
$\mathfrak{m}^{\mathbb{C}}$:
$\mathfrak{m}^{\mathbb{C}}=\mathfrak{m}^{+}\oplus \mathfrak{m}^{-}$,
where $\mathfrak{m}^{\pm}$ are $(\pm\sqrt{-1})$-eigenspaces of the
complex structure $J$. Since the group $K$ of isometries on $M^n$ and $M^c$
preserves $J$, $\mathfrak{m}^+$ and $\mathfrak{m}^-$ are invariant
under the adjoint action of $K$, i.e.
$[\mathfrak{k},\mathfrak{m}^+]\subset \mathfrak{m}^+$,
$[\mathfrak{k},\mathfrak{m}^-]\subset \mathfrak{m}^-$, we have the
following lemma \cite{Mok}:

\begin{lem}
$\mathfrak{m}^+$ and $\mathfrak{m}^-$ are abelian subalgebras of
$\mathfrak{g}^{\mathbb{C}}$. Moreover, the complex vector subspace
$\mathfrak{p}=\mathfrak{k}^{\mathbb{C}}+\mathfrak{m}^- \subset
\mathfrak{g}^{\mathbb{C}}$ is a complex Lie subalgebra of
$\mathfrak{g}^{\mathbb{C}}$.
\end{lem}

Denote $P=\exp (\mathfrak{p})$ to be the real Lie subgroup of
$G^{\mathbbm{C}}$ corresponding to $\mathfrak{p}$.

\begin{thm}
(Borel embedding theorem).

The embedding $G^c\hookrightarrow G^{\mathbb{C}}$ induces a
biholomorphism $M^c \cong G^c / K \hookrightarrow G^{\mathbb{C}} /P$
onto the complex homogeneous manifold $G^{\mathbb{C}} / P$. The
embedding $G^n\hookrightarrow G^{\mathbb{C}}$ induces an open
embedding $M^n\cong G^n / K \hookrightarrow G^{\mathbb{C}} / P \cong
M^c$, realizing $M^n$ as an open subset of its compact dual $M^c$.
\end{thm}

We will use \textbf{\textit{b}} to denote the Borel embedding in the Hermitian
symmetric space case. Since we have
 $M^c\cong G^{\mathbb{C}} /P$,
we can define an embedding $\textbf{\textit{g}}: G^n/K \hookrightarrow G^{\mathbb{C}}/P \cong  G^c/K$ by $AK \mapsto AP \mapsto ABK$, where
$B \in P$ such that $AB\in G^c$. It is obviously that \textbf{\textit{b}}=\textbf{\textit{g}} and
\begin{lem}
The embedding \textbf{\textit{b}}=\textbf{\textit{g}} is $K$-equivariant.
\end{lem}
\begin{proof}
For any $AK \in G^n/K$, $\textbf{\textit{g}}(AK)=ABK$,
where $B\in P$ such that $AB\in G^c$. Now for any $k\in K$,
$\textbf{\textit{g}}(k\cdot AK)=kAB^{'}K$, where $B^{'}\in P$ such that
$kAB^{'}\in G^c$. We can just choose $B^{'}=B$. So $\textbf{\textit{g}}(k\cdot AK)=kABK=k\cdot \textbf{\textit{g}}(AK)$.
\end{proof}

\begin{eg}
For the Riemann sphere $S^2=\text{SU}(2)/\text{S}(\text{U}(1)^2)$, its noncompact dual is
the unit disk $\text{SU}(1,1)/\text{S}(\text{U}(1)^2)$ in the complex line $\mathbb{C}$.
The Borel embedding is defined through the manifold
$\text{SL}(2,\mathbbm{C})/P$ which is isomorphic to $
\text{SU}(2)/\text{S}(\text{U}(1)^2)$, where $P$ is a parabolic
subgroup of $\text{SL}(2,\mathbbm{C})$.
Since we can embed the complex line into the sphere by stereographic projection,
the unit disk is the lower hemisphere of the sphere, this is just the Borel embedding. In this case,
the image of the Borel embedding consists of those points with distance less than half of the diameter (least uppon
bound of the length of an arbitrary minimal geodesic) of the sphere
from the south pole.
\end{eg}

For Hermitian symmetric spaces, We also have polydisc theorem and Harish-Chandra embedding theorem which we will not discuss
in this paper.

\bigskip

\subsection{The generalized embedding for symmetric R-spaces}
From the above two subsections, we have space-like
embedding \textbf{\textit{p}} for the generalized Grassmannians and Borel embedding \textbf{\textit{b}} for the Hermitian symmetric spaces.
Though these two embeddings seem very different from
their definitions, their formulas are the same (the same \textbf{\textit{g}}) when we write them from the Lie group theoretical aspect.
The definition of \textbf{\textit{g}} depends  on the fact that the compact symmetric space $G/K$
admits a connected transformation group larger than the isometry group $G$, which is just the definition for symmetric R-spaces.

Symmetric R-spaces, introduced by Takeuchi and Nagano in the 1960s, form a
distinguished subclass of compact Riemannian symmetric spaces. A symmetric R-space is a compact symmetric space
$G/K$ which admits a connected transformation group larger than the isometry group $G$.
Let $L$ be the largest such group, then $L$ is semisimple and $G/K \cong L/P$ where
$P$ is a parabolic subgroup of $L$ and $G \bigcap P=K$. Irreducible symmetric R-spaces
have been first classified by Kobayashi and Nagano in \cite{KN}\cite{Na}:

(i) all Hermitian symmetric spaces of compact type;

(ii) Grassmannian manifolds $\text{O}(n+m)/
\text{O}(n)\text{O}(m)$, $\text{Sp}(n+m)/
\text{Sp}(n)\text{Sp}(m)$;

(iii) the classical groups $\text{SO}(m)$, $\text{U}(m)$, $\text{Sp}(m)$;

(iv) $\text{U}(2m)/
\text{Sp}(m)$, $\text{U}(m)/ \text{O}(m)$;

(v) $\text{SO}(n+1)\times \text{SO}(m+1)/
S(\text{O}(n) \times \text{O}(m))$;

(vi) the Cayley projective plane and three exceptional spaces.

From the above list, we can see that all the compact generalized Grassmannians and compact Hermitian symmetric spaces are symmetric R-spaces.

\begin{rmk}
Note that there are three types irreducible classical symmetric R-spaces which are not generalized
Grassmannian: $SO(n+2)/SO(n)SO(2)$, $SO(n+1)\times SO(m+1)/S(O(n)\times O(m))$, and $SO(2n)/U(n)$ (n odd). However, we can also describe
them as some kind of ``Grassmannian''.

For $SO(n+2)/SO(n)SO(2)$, it is a Hermitian symmetric space, consisting of \textbf{oriented} subspaces $L \cong \mathbb{R}^2$ in $\mathbb{R}^{n+2}$.

For $SO(n+1)\times SO(m+1)/S(O(n)\times O(m))$, it is a real form of $SO(n+m+2)/SO(n+m)SO(2)$, hence consisting of \textbf{oriented} subspaces
$L \cong \mathbb{R}^2$ in $\mathbb{R}^{n+m+2}$ which are invariant under the corresponding isometric involution.

For $SO(2n)/U(n)$ (n odd), it is a reflective submanifold of $U(2n)/U(n)U(n)$, hence consisting of subspace
$L \cong \mathbb{C}^n$ in $\mathbb{C}^{2n}$ which are invariant under the corresponding isometric involution.
\end{rmk}

\begin{rmk}
From Remark 3.2, we can also describe these three types classical symmetric R-spaces as some kind of ``Grassmannians'', the natural question is whether their noncompact duals
are corresponding ``space-like Grassmannians''? However, this is not true in general.

For  $SO(n+2)/SO(n)SO(2)$, consisting of \textbf{oriented} subspaces $L \cong \mathbb{R}^2$ in $\mathbb{R}^{n+2}$, its noncompact dual is $SO_{o}(2,n)/SO(n)SO(2)$, consisting of space-like $L \cong \mathbb{R}^2$ in $\mathbb{R}^{n+2}$ without the oriented condition.

For $SO(n+1)\times SO(m+1)/S(O(n)\times O(m))$, same with the $SO(n+2)/SO(n)SO(2)$ case, since it is a real form of $SO(n+m+2)/SO(n+m)SO(2)$.

For $SO(2n)/U(n)$ (n odd), it is a reflective submanifold of $U(2n)/U(n)U(n)$, hence its noncompact dual is the corresponding space-like Grassmannian.
That means the space-like embedding  \textbf{\textit{p}} can be generalized to $SO(2n)/U(n)$ (n odd).
\end{rmk}

From this definition or characterization for a symmetric R-space, we can easily construct an open embedding from its noncompact dual to it.
Let $M^c=G^c/K \cong L/P$ be a \textbf{symmetric R-space}, and $M^n=G^n/K$ its
noncompact dual, then $G^n$ is a subgroup of $L$ and $G^n \bigcap P=K$. We define
\[
\textbf{\textit{g}}: G^n/K \hookrightarrow L/P \cong  G^c/K
\]
\[
AK\mapsto AP\mapsto ABK
\]
where $B \in P$ such that $AB \in G^c$ and such $B$ is unique up to $K$. That is, this map generalizes
space-like embedding \textbf{\textit{p}} for the generalized Grassmannians and
Borel embedding \textbf{\textit{b}} for the Hermitian symmetric spaces. Obviously, we have
\begin{lem}
The embedding \textbf{\textit{g}} is $K$-equivariant.
\end{lem}

Since the duality
between compact symmetric space and noncompact symmetric space is
invariant under reflective submanifold, we have the following proposition.

\begin{prop}
For any symmetric R-space $M$, if $N$ is a reflective submanifold of $M$
which is also a symmetric R-space, then $\textbf{\textit{g}}_N=\textbf{\textit{g}}_M|_N$.
\end{prop}

Since this embedding \textbf{\textit{g}} generalizes the embeddings \textbf{\textit{p}} and \textbf{\textit{b}},
we will call it the generalized embedding in the rest of the paper.

As a conclusion of this section, the three embeddings
\textbf{\textit{p}}, \textbf{\textit{b}}  and \textbf{\textit{g}} are compatible with each other, they are all $K$-equivariant
and satisfy Proposition 3.1.

\bigskip

\section{The embeddings from Lie algebra aspect}
In the above section, we have a generalized embedding \textbf{\textit{g}} for symmetric R-spaces which generalizes both \textbf{\textit{p}} and \textbf{\textit{b}}, defined from Lie group theoretical aspect.
How to characterize the image of this embedding?
To answer this question, we will first construct an embedding for symmetric R-spaces from Lie algebra theoretical aspect, so the image is easy to describe, and then
show this new embedding is just our \textbf{\textit{g}}.

In this section, we will first give further characterizations and properties of symmetric R-spaces,  then review the cut loci of symmetric R-spaces, and finally construct an embedding from the Lie algebra
aspect using the cut loci.

\subsection{Unit lattices of symmetric R-spaces}
We will list some
important properties of symmetric R-spaces in this subsection.

\bigskip

In 1984, Takeuchi \cite{Ta} showed that irreducible
symmetric R-spaces are either irreducible Hermitian symmetric
spaces of compact type or compact connected real forms of them and vice-versa.

\bigskip

In 1985, O. Loos \cite{Lo} gave another intrinsic characterization of symmetric R-spaces among
all compact symmetric spaces by the property that its unit lattice is \textbf{orthonormal} (i.e. the unit lattice possesses an orthonormal
$\mathbb{Z}$-basis). Recall that for a compact symmetric space $G/K$ with the corresponding Cartan decomposition $\mathfrak{g}=\mathfrak{k}\bigoplus\mathfrak{m}$,
choose a maximal abelian subspace  $\mathfrak{a} \subset \mathfrak{m} $ in $\mathfrak{m}$ (called Cartan subalgebra of $(G,K)$), then the unit lattice is:
\[
\Gamma(G,K):=\{A \in \mathfrak{a}: \exp A\in
K\}.
\]
The unit lattice $\Gamma(G,K)$ of $G/K$ is said to be \textbf{orthonormal} if
there exits a basis $\{A_1, \cdots, A_r\}$ of $\mathfrak{a}$ (where $r=\mathrm{rank}(M):=\dim(\mathfrak{a})$)
with the properties

(i) $A_i \bot A_j$ if $i\neq j$

(ii) $|A_i|=|A_j|$

(iii) $\Gamma(G,K)=span_{\mathbb{Z}}(A_1, \cdots, A_r)=\{\sum_{i=1}^{i=r} m_iA_i; m_i \in \mathbb{Z}\}$

We will compute some examples of unit lattices of compact symmetric spaces below.

\begin{eg}The Real Grassmannian $M=\text{O}(n+m)/
\text{O}(n)\text{O}(m)$ $(m\geqslant n)$,  then
$M\cong\{L\subset\mathbb{R}^{n+m}: L\cong \mathbb{R}^n\}$ and $\mathrm{rank}(M)=n$.

For the Cartan decomposition
$\mathfrak{g}=\mathfrak{k}+\mathfrak{m}$, we have

$\mathfrak{g}=\underline{o}(n+m)=\left \{\left(
\begin{array}
[c]{cc}%
A & B\\
-B^{t} & D
\end{array}
\right):A^{t}+A=0, D^{t}+D=0 \right \}$,

$\mathfrak{k}=\underline{o}(n)\times\underline{o}(m)=\left \{\left(
\begin{array}{cc}
A & 0 \\
0 & D
\end{array}
\right):A^{t}+A=0, D^{t}+D=0\right \}$,

$\mathfrak{m}=\left \{\left(
\begin{array}
[c]{cc}%
0 & B\\
-B^{t} &0
\end{array}
\right)\right \}$.

Denote $R_{ij}$ $(1\leq i\leq n, n+1\leq j \leq n+m)$ to
be the matrix with 1 at $(i,j)$-entry, -1 at $(j,i)$-entry, and 0
otherwise. Then $\{R_{ij}'s\}$ form an orthonormal basis of
$\mathfrak{m}$ with respect to the inner product $\langle
\cdot,\cdot \rangle$, where $\langle X,Y
\rangle:=-\frac{1}{2}tr(XY)$. By computing the Lie bracket, we know
that $\mathfrak{a}:=span_{\mathbb{R}}(R_{1,n+1},R_{2,n+2},\cdots ,R_{n,2n})$ is a
Cartan subalgebra of $(G,K)$, then $\Gamma(G,K)=\{A \in
\mathfrak{a}|\exp A \in \text{O}(n)\times
\text{O}(m)\}=\{\sum_{i=1}^{n}m_i\pi R_{i,n+i}, m_i \in
\mathbb{Z}\}$. The unit lattice $\Gamma(G,K)$ is generated by
$\{A_1=\pi R_{1,n+1}, \cdots, A_n=\pi R_{n,2n}\}$ with $A_i \bot
A_j$ and $|A_i|=|A_j|$, $\forall i\neq j \in \{1,2, \cdots, n\}$,
hence orthonormal.
\end{eg}

Similar to the above example, we can check that the unit lattices of all the classical
irreducible symmetric R-spaces are orthonormal.

\begin{eg}
For the special unitary group $\text{SU}(3)$ (not a symmetric R-space), its unit lattice is just its integral lattice,
the picture of this lattice can be found in page 227 of \cite{BD}. Obviously, this lattice is
not orthonomal.
\end{eg}

For the exceptional symmetric spaces, because the explicit
description of the corresponding Lie algebras as matrix algebras is
too unwidely to be useful generally, we can't determine whether
these unit lattices are orthonormal or not by direct computations as above.
However for the five
exceptional compact Lie groups, we can use the relationship between
integral lattices and inverse roots to get the answer
(page 223 of \cite{BD}). That is, for the five simply-connected exceptional compact Lie groups,
the integral lattice is the same with
the abelian group generated by inverse roots, hence
they are not orthonormal (the list of
the root system of exceptional Lie algebras can be found in page 686-692
of \cite{K}).

\bigskip

In 2012, P. Quast and M.S. Tanaka \cite{QT} gave the following important property of symmetric R-spaces: reflective submanifolds of symmetric R-spaces
are (geodesically) convex. That means any shortest geodesic arc in the reflective submanifolds is still shortest in the symmetric R-spaces.

Further characterizations and properties of symmetric R-spaces can be found in \cite{Ta2}.
\bigskip

\subsection{Cut loci of symmetric R-spaces}
This subsection is mainly based on T.Sakai's work \cite{TS1}\cite{TS2}\cite{TS3}.
More about cut loci, Cartan polyhedrons and injective radii for symmetric spaces
can be found in \cite{Y1}\cite{Y2}. Let $(M,g)$ be a
complete Riemannian manifold with a fixed point $x$ in $M$, then we
have an exponential map $\exp_{x}:T_{x}M\rightarrow M$. For any
$X\in T_{x}M, |X|=1, \gamma_{X}(t):=\exp_{x}(tX)$ is a geodesic
parametrized by arclength.

\begin{defn}
$t_{0}X$ (resp. $\exp_{x}(t_{0}X)$) is called a tangent conjugate
point (resp. conjugate point) of $x$ along a geodesic $\gamma_X$ if
there exists a nonzero Jacobi field $J(t)$ along $\gamma_X$ such
that $J(0)=J(t_0)=0$.

$\bar{t}_{0}X$ (resp. $\exp_{x}(\bar{t}_{0}X)$) is called a \textbf{tangent cut point}
$($resp. \textbf{cut point}$)$ of $x$ along $\gamma_X$ if
$\gamma_{X|[0,\bar{t}_{0}]}$ is minimal but $\gamma_{X|[0,s]}$ is not
minimal for any $s>\bar{t}_{0}$.
\end{defn}

Denote $\mathrm{Cut}(x)$ to be the set of all cut points of $x$
along any geodesic, it is called the \textbf{cut loci} of $x$.
The following standard result about tangent cut points is
Proposition 2.1 in \cite{TS1}.

\begin{prop}
If $\bar{t}_0X$ is a tangent cut point of $x$ along $\gamma_X$, then
either $(a)$ $\bar{t}_0X$ is the first tangent conjugate point of
$x$ along $\gamma_X$,
or $(b)$ there exists a geodesic $\gamma_Y \neq \gamma_X$ joining
$x$ to $\gamma_X(\bar{t}_0)$ such that
$l(\gamma_Y|[0,\bar{t}_0])=l(\gamma_X|[0,\bar{t}_0])$. Here $l$ is the
length of the geodesic.

Conversely, if $(a)$ or $(b)$ is satisfied, then there exists
$\bar{t} \in (0,\bar{t}_0]$, such that $\bar{t}X$ is a tangent cut
point of $x$ along $\gamma_X$.
\end{prop}

\begin{eg}
The three space forms, i.e. simply connected Riemannian manifold with constant
sectional curvature $K$,

$K=+1$, unit sphere $S^n$ has $\mathrm{Cut}(x)=\{-x\}$.

$K=0$, Euclidean space $\mathbb{R}^n$ has
$\mathrm{Cut}(x)=\emptyset$.

$K=-1$, hyperbolic space $\mathbb{H}^n$ has
$\mathrm{Cut}(x)=\emptyset$.
\end{eg}

\begin{eg}
Real projective space $\mathbb{RP}^n$ has
$\mathrm{Cut}(x)\cong \mathbb{RP}^{n-1}$.

Complex projective space $\mathbb{CP}^n$ with Fubini-study metric has
$\mathrm{Cut}(x)\cong \mathbb{CP}^{n-1}$.
\end{eg}

\begin{eg}
Flat torus $T^2=S^1\times S^1$, $\mathrm{Cut}(x)\cong S^1\times\{0\}\cup\{0\}\times
S^1$.
\end{eg}

\begin{rmk}
All the examples above are homogeneous spaces, hence the cut loci
of two different points look the same.
\end{rmk}

Let $M$ be a symmetric space.
If $M$ is
noncompact, then the sectional curvature of $M$ is non-positive, by
Hadamard-Cartan Theorem, the exponential map is a covering map. For
any $x,y \in M$, there exists a unique minimizing geodesic joining
$x$ and $y$. So $\mathrm{Cut}(x)=\emptyset$ for any $x \in M$. If
$M$ is compact, its diameter is finite, hence there exists a cut
point for any point $x\in M$ along any geodesic starting from $x$.
So $\mathrm{Cut}(x)\neq\emptyset$ for any $ x \in M$.

Now given a compact symmetric space $M$ and a fixed point $o$ in
$M=G/K$, we want to determine the tangent cut loci of $o$. For the corresponding
Lie algebra $\mathfrak{g}$, $\mathfrak{k}$ of $G$, $K$, we have the
Cartan decomposition $\mathfrak{g}=\mathfrak{k}+\mathfrak{m}$, where
$T_oM\cong \mathfrak{m}$. Let $\mathfrak{a}$ be a maximal abelian
subspace of $\mathfrak{m}$,
since $\mathrm{Ad}K(\mathfrak{a})=\mathfrak{m}$ and
$\mathrm{Ad}K$ acts on $\mathfrak{m}$ as isometries, it suffices to
compute $\bar{t}_0(X)$ for $X \in \mathfrak{a}$. From Proposition 3.1, obviously, we have

\begin{prop}
Let $M=G/K$ be a compact symmetric space and $\mathfrak{a}$ be a
Cartan subalgebra of $(G,K)$. For a unit vector $X\in \mathfrak{a}$,
assume that $\bar{t}_0 X$ is a tangent cut point of $o$ along
$\gamma_X$. Then either $\bar{t}_0 X$ is the first tangent conjugate
point of $o$ along $\gamma_X$ or there exists a unit vector $Y \in
\mathfrak{a}$, $Y\neq X$ such that $\exp_o \bar{t}_0 X=\exp_o
\bar{t}_0 Y$ does hold.
\end{prop}

Now given
a unit tangent vector $X \in \mathfrak{a}$, $|X|=1$, we want to
compute $\bar{t}_0(X)$ to determine the cut point of $o$ along the
geodesic $\gamma_X$. By the above proposition,
$\bar{t}_0(X)=\min\{t_0(X), \tilde{t}_0(X)\}$, where $t_0(X)$ is
given by the first conjugate point of $o$ along $\gamma_X$,
$\tilde{t}_0(X)$ is the minimum positive value such that
$\exp_o\tilde{t}_0(X)X=\exp_o\tilde{t}_0(X)Y$ holds for some unit
tangent vector $Y\in \mathfrak{a}$, $Y\neq X$.

\begin{rmk}
For simply-connected Riemannian symmetric spaces, the first conjugate loci
and the cut loci coincide, without the simply-connectedness condition, we have
$\tilde{t}_0(X)\leq t_0(X)$ (see the proof of Theorem 2.5 in \cite{TS1}).
\end{rmk}

\begin{lem}
Let $X=Adk(H) \in \mathfrak{m}$ with $H \in \mathfrak{a}$, $k \in
K$, $|H|=1$. Then
\[
\bar{t}_0(X)=\bar{t}_0(H)=\min_{\{A \in
\Gamma(G,K)-\{0\}\}}\frac{\langle A,A \rangle }{2|\langle H,A
\rangle|},
\]
where $\Gamma(G,K):=\{A \in \mathfrak{a}: \exp A\in
K\}$ is the \textbf{unit lattice}.
\end{lem}

Now let $H$ be a compact Lie group, if we put $G=H\times H$,
$K=\{(x,x)|x \in H\}$ and the involution $\sigma$ of $G$:
$\sigma(x,y)=(y,x)$, then $(G,K)$ is a compact Riemannian symmetric
pair corresponding to $H \cong G/H$. Let $\mathfrak{h}$ be the Lie algebra of $H$ and
$\mathfrak{a}^{*} \subset \mathfrak{h}$ be a maximal abelian
subalgebra of $\mathfrak{h}$, then $\mathfrak{m}=\{(X,-X):X \in
\mathfrak{h}\}$ and $\mathfrak{a}:=\{(Y,-Y): Y \in
\mathfrak{a}^{*}\}$ is a Cartan subalgebra of $(G,K)$.

\begin{lem}
Let $X \in \mathfrak{a}^{*}$ with $|X|=1$, then
$\bar{t}_0(X)=\min_{\{A \in \Gamma(H)-\{0\}\}}\frac{\langle A,A
\rangle }{2|\langle X, A \rangle|}$, where $\Gamma(H):=\{A \in
\mathfrak{a}^{*}: \exp A=e\}$.
\end{lem}

Note that the unit lattice $\Gamma(H):=\{A \in \mathfrak{a}^{*}: \exp A=e\}$ is just the
\textbf{integral lattice} of the compact Lie group $H$.

By the above lemmas, we have the following theorem.
\begin{thm}
Let $M=G/K$ be a compact symmetric space. Then the tangent cut loci
of $o$ is given by $\mathrm{Ad}K(\bigcup\{\bar{t}_0(X)X;X \in
\mathfrak{a}, |X|=1\})$, i.e. the tangent cut loci of $o$ is
determined by the tangent cut loci of $o$ in the flat torus
$\mathfrak{a} / \Gamma(G,K)$.
\end{thm}

If the unit lattice $\Gamma(G,K)$ is orthonormal, then for any
unit vector $X=\sum x_iA_i$ in $\mathfrak{a}$, the tangent cut point
along direction $X$ is determined by $\bar{t}_0(X)=1/(2\max|x_i|)$. In fact, if we put $\alpha^2=\alpha_i^2=\langle A_i,A_i\rangle$,
then $\sum (x_i)^2(\alpha_i)^2=1$. For any non-zero $A=\sum m_iA_i \in
\Gamma(G,K)$, we have
\begin{align*}
\frac{\langle A,A \rangle}{2|\langle X,A\rangle |} &
=\frac{\alpha^2({m_1}^2+\cdots +{m_r}^2)}{2\alpha^2|m_1x_1+\cdots +m_rx_r|} \\
& \geq \frac{({m_1}^2+\cdots +{m_r}^2)}{2\max |x_i|(|m_1|+\cdots
+|m_r|)} \\
& \geq \frac{1}{2\max |x_i|}
\end{align*}
and if $\max|x_i|=|x_{i_{0}}|$, then for taking $m_i=\delta_{i,i_{0}}$, the $"="$ holds. Since
the unit lattice of a symmetric R-space is always orthonormal, we have the following theorem \cite{HN}:

\begin{thm}
Let $M=G/K$ be a symmetric R-space, then the tangent cut loci
of $o$ is given by $\mathrm{Ad}K(\bigcup\{1/(2\max|x_i|)X;X=\sum x_iA_i \in
\mathfrak{a}, |X|=1\})$.
\end{thm}

\bigskip

\subsection{The embedding using cut loci}
Let $M^c=G^c/K$ be a \textbf{symmetric R-space} and $M^n=G^n/K$ be its
noncompact dual. We want to
construct a natural embedding $M^n \hookrightarrow M^c$.
The exponential map of any noncompact symmetric space is a
diffeomorphism. Let $R:=\{tX \in T_oM^c| |X|=1, t <\bar{t}_0(X) \}
\subset T_{o}M^c$, where $\bar{t}_0(X)$ is a positive number such
that $\bar{t}_0(X)X$ is the tangent cut point of $o$ along
$\gamma_X$. Then $\mathrm{Int}(o):=M^c\backslash
\mathrm{Cut}(o)=\exp_o(R) \subset M^c$ is an open cell, with the
exponential map a diffeomorphism from $R$ to $\mathrm{Int}(o)$. Now
we only need to construct an injective map $h: T_oM^n \rightarrow
T_oM^c$ such that $\mathrm{Im}(h)\subset R$, then the map
$\textbf{\textit{f}}=\exp^c_o \circ h \circ (\exp^n_o)^{-1}$: $M^n \rightarrow
T_oM^n \rightarrow
T_oM^c \rightarrow M^c$ is an
open embedding.

\begin{eg}
The injective map from positive line to circle.

For any $\exp(\sqrt{-1}\theta) \in S^1$, $\exp(t) \in \mathbb{R}^+$, where
$\theta ,t \in \mathbb{R}$, the stereographic projection maps
$\exp(\sqrt{-1}\theta)$ to $\exp(t)$ if and only if
$\tan(\frac{\theta}{2})=\tanh(\frac{t}{2})$. So we can define an
injective map from $\mathbb{R}^+$ to $S^1$ as follows:
\[
i:\mathbb{R}^+ \rightarrow S^1.
\]
\[
\exp(t) \mapsto \exp(2\sqrt{-1}\arctan(\tanh\frac{t}{2})).
\]
\end{eg}

Denote the Cartan decompositions of the corresponding symmetric spaces as
$\mathfrak{g}_n=\mathfrak{k} \oplus \mathfrak{ m}_n$ for $M^n$ and
$\mathfrak{g}_c=\mathfrak{k} \oplus
\sqrt{-1}\mathfrak{m}_n=\mathfrak{k} \oplus \mathfrak{m}_c$ for $M^c$. Let
$\mathfrak{a}_n \subset \mathfrak{m}_n$,
$\mathfrak{a}_c=\sqrt{-1}\mathfrak{a}_n \subset
\sqrt{-1}\mathfrak{m}_n=\mathfrak{m}_c$ be the maximal abelian
subspaces of $\mathfrak{m}_n$, $\mathfrak{m}_c$ respectively.
We will use $X$ to represent vector in
$\mathfrak{m}_n$, and $\sqrt{-1}X$ to represent the corresponding
vector in $\mathfrak{m}_c$.

First we construct an injective map from $\mathfrak{a}_n$ to
$\mathfrak{a}_c$. As before, for the compact symmetric space $M^c$,
the unit lattice
$\Gamma(G^c,K)$ is a lattice in $\mathfrak{a}_c$ which is
generated by $\{\sqrt{-1}A_1, \cdots, \sqrt{-1}A_r\}$, where $r=\mathrm{rank}(M)$. Since $M^c$ is a symmetric R-space,
we have $A_i \bot A_j$ and $|A_i|=|A_j|$ for any $i \neq j \in \{1,2,
\cdots, r\}$.
Note that for any
$\sqrt{-1}X=(\sqrt{-1}x_1A_1, \cdots, \sqrt{-1}x_rA_r)$ in
$\mathfrak{a}_c$, $\sqrt{-1}X$ is a tangent cut point if and only if $\max|x_i|=1/2$.

Define:
\[
h:\mathfrak{a}_n \rightarrow \mathfrak{a}_c,
\]
\[
(x_1A_1, \cdots ,x_rA_r) \mapsto -(\frac{\arctan(\tanh{\pi x_1})}{\pi}\sqrt{-1}A_1,\cdots,
\frac{\arctan(\tanh{\pi x_r})}{\pi}\sqrt{-1}A_r),
\]
note here the function $a(x):=\arctan(\tanh{\pi x})/\pi$ is a monotone increasing odd function.
 Since $\arctan(\tanh{\pi x})/\pi \in (-\frac{1}{4},\frac{1}{4})$,
we have $Im(h)=\frac{R}{2} \bigcap \mathfrak{a}_c$, where
$\frac{R}{2}=\{\sqrt{-1}tX \in \mathfrak{m}_c||X|=1, t
<\frac{\bar{t}_0(\sqrt{-1}X)}{2}\}$.

Now we want to extend $h$ to the whole tangent space. For any
tangent vector $X \in \mathfrak{m}_n$, there exists
$k \in K$, such that
$X=\mathrm{Ad}k(H)$ for a unique $H \in \mathfrak{a}_n$. Then
we define $h^{'}(X)=\mathrm{Ad}k(h(\mathrm{Ad}k^{-1}X))$.

It is easy to see that such a $h^{'}$ does not depend on the choice of the maximal abelian subspace
$\mathfrak{a}_n$ and $k \in K$.
In fact, $\mathrm{Ad}k\cdot \mathfrak{a}_n=\mathfrak{a}_n^{'}$ is
also a maximal abelian subspace in $\mathfrak{m}_n$ and $X \in \mathfrak{a}_n^{'}$.
Consider
their corresponding compact duals $\mathfrak{a}_c$ and
$\mathfrak{a}_c^{'}=\mathrm{Ad}k\cdot \mathfrak{a}_c$, if
$\{\sqrt{-1}A_1, \cdots ,\sqrt{-1}A_r\}$ are generators of
$\Gamma(G^c,K)$, then $\{\sqrt{-1}A_1^{'}=\mathrm{Ad}k \cdot
\sqrt{-1}A_1, \cdots, \sqrt{-1}A_r^{'}=\mathrm{Ad}k \cdot
\sqrt{-1}A_r\}$ are generators of the lattice $\Gamma(G^c,K)^{'}$.
Write $X=\sum x_i^{'}A_i^{'}$, then under the map $h^{'}$ defined above:
\[
X=(x_1^{'}A_1^{'}, \cdots ,x_r^{'}A_r^{'})
\xrightarrow{\mathrm{Ad}k^{-1}} (x_1^{'}A_1, \cdots ,x_r^{'}A_r)
\]
\[\xrightarrow{h} -(\frac{\arctan(\tanh
{\pi x_1^{'}})}{\pi}\sqrt{-1}A_1,\cdots,
\frac{\arctan(\tanh{\pi x_r^{'}})}{\pi}\sqrt{-1}A_r),
\]
\[\xrightarrow{\mathrm{Ad}k} -(\frac{\arctan(\tanh
{\pi x_1^{'}})}{\pi}\sqrt{-1}A_1^{'},\cdots,
\frac{\arctan(\tanh{\pi x_r^{'}})}{\pi}\sqrt{-1}A_r^{'}),
\]
which has the same form with the map $h$ on $\mathfrak{a}_n$.

From above, we have a globally defined map, also denoted by
$h:\mathfrak{m}_n\rightarrow \mathfrak{m}_c$. We can now
construct an injective map from $M^n$ to $M^c$, $\textbf{\textit{f}}=\exp^c_o \circ h
\circ (\exp^n_o)^{-1}$: $M^n \rightarrow \mathfrak{m}_n \rightarrow
\mathfrak{m}_c \rightarrow M^c$.

Since $\mathrm{Im}(h)=\frac{R}{2}
\subset \mathfrak{m}_c$, $\textbf{\textit{f}}$ is injective and obviously

\begin{prop}
$\mathrm{Im}(\textbf{\textit{f}})=\exp_o^c(\frac{R}{2})$.
\end{prop}

Similar to the generalized embedding \textbf{\textit{g}}, the embedding \textbf{\textit{f}}
is also $K$-equivariant.

\begin{lem}
The embedding \textbf{\textit{f}} is $K$-equivariant.
\end{lem}

\begin{proof}
For any $x=e^XK \in M^n$ and $k \in K$,
\[
h(\mathrm{Ad}k\cdot X)=\mathrm{Ad}k \cdot h(X),
\]
\[
\textbf{\textit{f}}(k\cdot x)=\textbf{\textit{f}}(e^{\mathrm{Ad}k\cdot X}K)=e^{h(\mathrm{Ad}k\cdot X)}K
=e^{\mathrm{Ad}k\cdot h(X)}K=k \cdot e^{h(X)}K=k \cdot \textbf{\textit{f}}(x).
\]

\end{proof}

The following proposition is the key argument for the rest of the paper.

\begin{prop}
For any symmetric R-space $M$, if $N$ is a reflective submanifold of $M$
which is also a symmetric R-space, then $\textbf{\textit{f}}_N=\textbf{\textit{f}}_M|_N$.
\end{prop}

\begin{proof}
Let $M=G/K$ and $\rho:
M\rightarrow M$ be the isometric involution whose fixed point set is $N$,
then $N=H/I$, where $H \subset
G^{\tilde{\rho}}$, $I=H \bigcap K$, $\tilde{\rho}: G \rightarrow G$, $g \mapsto \rho g \rho^{-1}$.
For the corresponding Cartan decomposition, $\mathfrak {g}=\mathfrak{k}\oplus \mathfrak{m}$, $\mathfrak {h}=\mathfrak{i}\oplus \mathfrak{n}$,
$\mathfrak{m} \cong T_{o}M$, $\mathfrak{n}=\mathfrak{m}^{d\tilde{\rho}} \cong T_{o}N$.

Now take $\mathfrak{a}^{'} \subset \mathfrak{n}$ to be a maximal
abelian subspace in $\mathfrak{n}$, extend $\mathfrak{a}^{'}$ to
$\mathfrak{a}$, which is maximal abelian in $\mathfrak{m}$. Define
$\Gamma(G,K):=\{A \in \mathfrak{a}: \exp A \in K\}$ and
$\Gamma(H,I):=\{A^{'} \in \mathfrak{a}^{'}: \exp A^{'} \in I\}$.
Since $I\subset K ~  and  ~ \mathfrak{a}^{'} \subset \mathfrak{a}$,
$\Gamma(H,I) \subset \Gamma(G,K) \bigcap \mathfrak{a}^{'}$;
since $I=H\bigcap K$,
$\Gamma(H,I) \supset \Gamma(G,K) \bigcap \mathfrak{a}^{'}$.
So we have $\Gamma(H,I)=\Gamma(G,K) \bigcap \mathfrak{a}^{'}$.

By assumption, $M$ is a symmetric R-space, its unit lattice is orthonormal, i.e.
$\Gamma(G,K)=span_{\mathbb{Z}}(e_1, \cdots, e_r)$ with $\{e_1, \cdots, e_r\}$ orthonormal
basis of $\mathfrak{a}$ and $r=\mathrm{rank}(M)$. We want to compute $\Gamma(H,I)$. For
any $X \in \Gamma(G,K)$, we have $X \in \Gamma(H,I)$ if and only if $d\tilde{\rho}(X)=X$.
Consider the map $d\tilde{\rho}: \mathfrak {g} \rightarrow \mathfrak {g}$, by Observation 4 in \cite{QT}, $\mathfrak{a}$
is invariant under $d\tilde{\rho}$. Since $d\tilde{\rho}$ is the differential of an isometric involution that leaves
$\mathfrak{a}$ invariant, $d\tilde{\rho}|_{\mathfrak{a}}$ is an orthogonal transformation of $\mathfrak{a}$ that squares identity and hence preserves the
unit lattice $\Gamma(G,K) \in \mathfrak{a}$. By Proposition 5 in \cite{QT}, we can choose $\{e_1, \cdots, e_r\}$ such that
\begin{align*}
& (i)  ~ d\tilde{\rho}(e_{2j})=e_{2j-1}  ~ for  ~  1\leq j \leq p \\
& (ii)  ~ d\tilde{\rho}(e_{j})=e_{j}  ~ for  ~  2p+1\leq j \leq q \\
& (iii)  ~ d\tilde{\rho}(e_{j})=-e_{j}  ~ for  ~  q+1\leq j \leq r,
\end{align*}
where $0 \leq 2p \leq q \leq r$.

Then $\Gamma(H,I)=span_{\mathbb{Z}}(e_1+e_2, \cdots, e_{2p-1}+e_{2p}, e_{2p+1}, \cdots, e_q)$.

With the assumption that $N$ is also a symmetric R-space,  $\Gamma(H,L)$ must be orthonormal. That means $\Gamma(H,I)$
must either be $span_{\mathbb{Z}}(e_1+e_2, \cdots, e_{2p-1}+e_{2p})$ or $span_{\mathbb{Z}}(e_1, \cdots, e_q)$, where $2p\leq r$, $q\leq r$.

When $\Gamma(H,I)=span_{\mathbb{Z}}(e_1+e_2, \cdots, e_{2p-1}+e_{2p})$, the map $\textbf{\textit{f}}_N$ is induced by $h_N$
(denote $\mathfrak{a}_n^{'}$, $\mathfrak{a}_n$ as $\sqrt{-1}\mathfrak{a}_n^{'}=\mathfrak{a}^{'}$, $\sqrt{-1}\mathfrak{a}_n=\mathfrak{a}$, and $\tilde{e}_i \in \sqrt{-1}\mathfrak{a}_n$ such that $\sqrt{-1}\tilde{e}_i=e_i$)
\[
h_N:\mathfrak{a}_n^{'} \rightarrow \mathfrak{a}^{'},
\]
\[
(x_1(\tilde{e}_1+\tilde{e}_2), \cdots ,x_p(\tilde{e}_{2p-1}+\tilde{e}_{2p})) \mapsto
-(a(x_1)(e_1+e_2),\cdots,
a(x_p)(e_{2p-1}+e_{2p})),
\]
which is just the restriction of $h_M$:
\[
h_M:\mathfrak{a}_n \rightarrow \mathfrak{a},
\]
\[
(x_1\tilde{e}_1,\cdots ,x_r\tilde{e}_r) \mapsto
-(a(x_1)e_1,\cdots,
a(x_r)e_r),
\]
hence $\textbf{\textit{f}}_N=\textbf{\textit{f}}_M|_N$.

When $\Gamma(H,I)=span_{\mathbb{Z}}(e_1, \cdots, e_q)$, it is obviously that $\textbf{\textit{f}}_N=\textbf{\textit{f}}_M|_N$.
\end{proof}

\bigskip

\section{Comparing these embeddings}
We will first compare \textbf{\textit{f}} defined above with the space-like embedding \textbf{\textit{p}},
then compare \textbf{\textit{f}} with the Borel embedding \textbf{\textit{b}}, and finally compare \textbf{\textit{f}} with
the generalized embedding \textbf{\textit{g}}. The images of these embeddings are described using cut loci.

\bigskip

\subsection{Comparing \textbf{\textit{f}} with \textbf{\textit{p}}}
\begin{eg}
The Real Grassmannian $M=\text{O}(n+m)/
\text{O}(n)\text{O}(m)$ $(m\geqslant n)$, following the notations in Example 4.1,
$\mathfrak{a}=span_{\mathbb{R}}(R_{1,n+1},\cdots ,R_{n,2n})$ is a Cartan
subalgebra of $(G,K)$ and $\Gamma(G,K)=\{A \in \mathfrak{a}|\exp A
\in \text{O}(n)\times \text{O}(m)\}=\{\sum_{i=1} ^n m_i\pi
R_{i,n+i}, m_i \in \mathbb{Z}\}$. For any $ X=\sum_{i=1}^n x_i R_{i,n+i} \in \mathfrak{a}$,
$\sum_{i=1}^ {n}x_i^2=1$, we have
$\bar{t}_0(X):=\min_{\{A \in \Gamma(G,K)-\{0\}\}}
\frac{\langle A,A \rangle}{2|\langle X,A \rangle|}=\frac{\pi}{2\max
|x_i|}$. Put

\begin{align*}  & o:=\{(u_1,\cdots ,u_n,0,\cdots,
0)^{t}:u_i \in \mathbb{R}\} \cong \mathbb{R}^n \\ & o^\bot:=\{(0,\cdots,
0,u_{n+1},\cdots, u_{n+m})^{t}:u_i \in \mathbb{R}\} \cong \mathbb{R}^m \\
& W_l:=\{L \in M^c: dim(L\bigcap o^\bot)=l\}, (l=0,1,\cdots, n).
\end{align*}

\textbf{Fact 1}: Cut loci of $o$ in $M^c=\text{O}(n+m)/
\text{O}(n)\text{O}(m)$ is given by $V=W_1\bigcup \cdots \bigcup
W_n=\{L \in M^c: L\bigcap o^\bot \neq \{0\}\}$.

\begin{proof}
For any $ X=\sum_{i=1}^{n}x_i R_{i,n+i} \in \mathfrak{a}$,
$\sum_{i=1}^{ n}x_i^2=1$.

$\exp(tX)\cdot o=\{(u_1\cos tx_{1} , \cdots, u_n\cos tx_{n}
,-u_1\sin tx_{1} , \cdots, -u_{n}\sin tx_{n} ,0,\cdots, 0)^{t}:u_i
\in \mathbb{R}\}$.

For any $0\leq t< \bar{t}_0(X)=\frac{\pi}{2\max |x_i|}$, $\cos tx_{i}\neq0$ for any $i$,
$\exp(tX)\cdot o \in W_0$.

Cut loci of $o$ in $M^c$ is given by all $\exp(\bar{t}_0X)\cdot o$, when $\max|x_i|=|x_{i_1}|=\cdots =|x_{i_k}|> |x_{i_{k+1}}|\geq
\cdots \geq |x_{i_n}|$, $\exp(\bar{t}_0X)\cdot o \in W_k$.
\end{proof}

\textbf{Fact 2}: $M^+=\exp_o^c(\frac{R}{2})$.
\begin{proof}
Recall that $R:=\{tX \in T_oM^c| |X|=1, t <\bar{t}_0(X) \}
\subset T_{o}M^c$, For any $ X=\sum_{i=1}^n x_i R_{i,n+i} \in \mathfrak{a}$, we have
\begin{align*}
0\leq t< \frac{\bar{t}_0(X)}{2}=\frac{\pi}{4\max |x_i|}
& \Longleftrightarrow |tx_i|<\frac{\pi}{4}, ~   ~ \forall i \\
& \Longleftrightarrow (\cos tx_i)^2 > (\sin tx_i)^2 , ~  ~ \forall i \\
& \Longleftrightarrow \exp(tX)\cdot o  ~ is ~ space-like.
\end{align*}

Since the embedding is $K$-equivariant, we have $\exp_o^c(\frac{R}{2})= M^+$.
\end{proof}

That is, for the Real Grassmannian, the following is true:

$ M^n=\{L \subset V\}^+= \exp_o^c(\frac{R}{2})\subset M^c=\{L \subset
V\}=\exp_o^c(R)\bigcup \mathrm{Cut}(o)$.
\end{eg}

The generalization of the results for real Grassmannians is as follows:

\begin{prop}
For any compact generalized Grassmannian $M^c$ and its noncompact dual
$M^n$, the image of the space-like embedding \textbf{\textit{p}} is
$\exp_o^c(\frac{R}{2})$.
\end{prop}
\begin{proof}
Any compact generalized Grassmannian $M^c$ is the fixed point set of
$\text{O}(m+n)/\text{O}(m)\text{O}(n)$ by some isometric involutions
$\sigma_{w_i}$, $i=1,\cdots ,s$, where $\sigma_{w_i}'s$ are the
involutions on $\text{O}(m+n)/\text{O}(m)\text{O}(n)$ defined by
some quadratic forms $w_i$ on it. i.e. $M^c=\bigcap
(\text{O}(m+n)/\text{O}(m)\text{O}(n))^{\sigma_{w_i}}$. For
convenience, use $R_o$ to denote $R$ for $M^c=\text{O}(m+n)/
\text{O}(m)\text{O}(n)$. By Lemma 7 in \cite{QT}, we have
\begin{align*}
\exp_o^c(\frac{R}{2}) & =\exp_o^c(\frac{R_o}{2}\bigcap\mathfrak{m}_c) \\
 &=\text{O}(m,n)/\text{O}(m)\text{O}(n) \bigcap M^c \\
 &=\{space-like ~ subspace ~ in ~ M^c\} \\
 &=M^n.
\end{align*}
\end{proof}

Now for $M^c$ a compact generalized Grassmannian with noncompact dual $M^n$,
the two embeddings $\textbf{\textit{p}},\textbf{\textit{f}}: M^n \hookrightarrow M^c$ are both $K$-equivariant and have the same image
$\mathrm{Im}(\textbf{\textit{p}})=\mathrm{Im}(\textbf{\textit{f}})=\exp_o^c(\frac{R}{2})$. We want to show
$\textbf{\textit{p}}=\textbf{\textit{f}}$ in fact.

\begin{eg}
$\text{O}(1,1)/\text{O}(1)^2 \hookrightarrow
\text{O}(2)/\text{O}(1)^2$.

For $\text{O}(1,1)/\text{O}(1)^2$, we have $\mathfrak{a}_n=\left
\{\left(
\begin{array}
[c]{cc}%
0 & t\\
t & 0
\end{array}
\right)\ : t\in \mathbb{R}\right \}=\mathfrak{m}_n$. After the
exponential map $\exp^n_o$, we get
$\text{O}(1,1)/\text{O}(1)^2=\left \{\left(
\begin{array}
[c]{cc}%
\cosh t & \sinh t\\
\sinh t & \cosh t
\end{array}
\right)\cdot K\ : t\in \mathbb{R}\right \}$. By viewing the manifold as a
Grassmannian, we can use $\left(
\begin{array}
[c]{cc}%
\cosh t & \sinh t\\
\sinh t & \cosh t
\end{array}\right)\cdot K$ to represent a dimension one linear subspace in
$\mathbbm{R}^2$ with basis $\left[
\begin{array}{cc}
\cosh t \\
\sinh t  \end{array} \right]=\left[
\begin{array}{cc}
1 \\
\tanh t \end{array} \right]$.

For $\text{O}(2)/\text{O}(1)^2$, we have $\mathfrak{a}_c=\left
\{\left(
\begin{array}
[c]{cc}%
0 & \theta\\
-\theta & 0
\end{array}
\right)\ : \theta\in \mathbb{R}\right \}=\mathfrak{m}_c$. After the
exponential map $\exp^c_o$, we get $\text{O}(2)/\text{O}(1)^2=\left
\{\left(
\begin{array}
[c]{cc}%
\cos \theta & \sin \theta\\
-\sin \theta & \cos \theta
\end{array}
\right)\cdot K\ : \theta\in \mathbb{R}\right \}$. By viewing the manifold
as a Grassmannian, we can use $\left(
\begin{array}
[c]{cc}%
\cos \theta & \sin \theta\\
-\sin \theta & \cos \theta
\end{array}\right)\cdot K$ to represent a dimension one linear subspace in
$\mathbbm{R}^2$ with basis $\left[
\begin{array}{cc}
\cos \theta \\
-\sin \theta  \end{array} \right]=\left[
\begin{array}{cc}
1 \\
-\tan \theta \end{array} \right]$.

Then $\left(
\begin{array}
[c]{cc}%
\cosh t & \sinh t\\
\sinh t & \cosh t
\end{array}\right)\cdot K$ and $\left(
\begin{array}
[c]{cc}%
\cos \theta & \sin \theta\\
-\sin \theta & \cos \theta
\end{array}\right)\cdot K$ represent the same linear subspace in
$\mathbbm{R}^2$ if and only if $\tanh t=- tan \theta$. Comparing
with the formula of \textbf{\textit{f}}, we see that the space-like
embedding \textbf{\textit{p}} is the same with the embedding \textbf{\textit{f}}.
\end{eg}

\begin{thm}
For generalized Grassmannians, the space-like
embedding \textbf{\textit{p}} and the embedding \textbf{\textit{f}} are the same.
\end{thm}
\begin{proof}
First we prove it for the real Grassmannians. Follow the
notations in Example 5.1, we have for any $X=\sum_{i=1}^{n}x_i
R_{i,n+i} \in \mathfrak{a}_c$,

$\exp(X)\cdot o=\{(u_1\cos x_{1} , \cdots, u_n\cos x_{n} ,u_1(-\sin
x_{1}) , \cdots, u_{n}(-\sin x_{n}) ,0,\cdots, 0)^{t}:u_i \in
\mathbb{R}\}$.

Similarly, for its noncompact dual, denote $\tilde{R}_{ij}$
$(1\leq i\leq n, n+1\leq j \leq n+m)$ to be the matrix with 1
at $(i,j)$-entry, 1 at $(j,i)$-entry, and 0 otherwise. We have for any $
Y=\sum_{i=1}^{n}y_i \tilde{R}_{i,n+i} \in \mathfrak{a}_n$,

$\exp(Y)\cdot o=\{(u_1\cosh y_{1} , \cdots, u_n\cosh y_{n} ,u_1\sinh
y_{1} , \cdots ,u_{n}\sinh y_{n} ,0,\cdots, 0)^{t}:u_i \in
\mathbb{R}\}$.

Then $\exp(X)\cdot o$ and $\exp(Y)\cdot o$ represent the same
subspace if and only if $\tanh y_i=- \tan x_i$ for each $i$. This
means the two maps \textbf{\textit{p}} and \textbf{\textit{f}} are the same on the maximal flat
$T:=exp(\mathfrak{a}_n)\cdot o$. Since both \textbf{\textit{p}} and \textbf{\textit{f}} are $K$-equivariant, they are the same.

For the other generalized Grassmannians, both the
embeddings \textbf{\textit{p}} and \textbf{\textit{f}} are just the restrictions of the embeddings in
the corresponding real Grassmannian case (Proposition 3.1 for \textbf{\textit{p}} and Proposition 4.3 for \textbf{\textit{f}}).
Since \textbf{\textit{p}} and \textbf{\textit{f}} coincide in the real
Grassmannian case, they are the same for all the generalized
Grassmannians.
\end{proof}

\bigskip

\subsection{Comparing \textbf{\textit{f}} with \textbf{\textit{b}}}
Now we compare the embedding \textbf{\textit{f}} with the Borel embedding \textbf{\textit{b}}
in the Hermitian symmetric space case. There are six types compact Hermitian symmetric spaces, namely, $\text{U}(n+m)/
\text{U}(n)\text{U}(m)$, $\text{SO}(2n)/ \text{U}(n)$,
$\text{Sp}(n)/ \text{U}(n)$, $\text{SO}(n+2)/
\text{SO}(n)\text{SO}(2)$, $\text{E}_6/
(\text{SPin}(10)\text{U}(1)/\mathbb{Z}_4)$ and $\text{E}_7/
(\text{E}_6\text{U}(1)/\mathbb{Z}_3)$.

\begin{thm}
For all the compact Hermitian symmetric spaces except $\text{SO}(n+2)/
\text{SO}(n)\text{SO}(2)$, the Borel embedding  \textbf{\textit{b}} and the
embedding \textbf{\textit{f}} are the same.
\end{thm}
\begin{proof}
For $\text{U}(n+m)/
\text{U}(n)\text{U}(m)$, $\text{SO}(2n)/ \text{U}(n)$,
$\text{Sp}(n)/ \text{U}(n)$, they are all generalized Grassmannians, we have \textbf{\textit{b}}=\textbf{\textit{g}}=\textbf{\textit{p}}.
And also \textbf{\textit{p}}=\textbf{\textit{f}},
hence \textbf{\textit{b}}=\textbf{\textit{f}}.
(Though $\text{SO}(2n)/ \text{U}(n)$ is not in the generalized Grassmannian list when $n$ is odd,  we can also define the space-like embedding \textbf{\textit{p}}
from Remark 3.3, and this \textbf{\textit{p}} is also the same with \textbf{\textit{g}}, \textbf{\textit{f}} obviously.)

For $M=\text{E}_6/
(\text{Spin}(10)\text{U}(1)/\mathbb{Z}_4)$ with $\text{rank}=2$. It
has a reflective submanifold $N=\text{U}(6)/ \text{U}(2)\text{U}(4)$
with $\text{rank}=2$ which is also a symmetric R-space. $M$ and $N$ have the same maximal flat torus $T$. Since $\textbf{\textit{b}}_N=\textbf{\textit{f}}_N$,  $\textbf{\textit{b}}_N=\textbf{\textit{b}}_M|_N$, $\textbf{\textit{f}}_N=\textbf{\textit{f}}_M|_N$
and $T \subset N$, we have $\textbf{\textit{b}}_M|_T=\textbf{\textit{f}}_M|_T$.
Since \textbf{\textit{b}} and \textbf{\textit{f}} are both $K$-equivariant, we have $\textbf{\textit{b}}_M=\textbf{\textit{f}}_M$.

For $M=\text{E}_7/
(\text{E}_6\text{U}(1)/\mathbb{Z}_3)$ with $\text{rank}=3$. It has a
reflective submanifold $N=\text{SO}(12)/ \text{U}(6)$ with
$\text{rank}=3$ which is also a symmetric R-space. The argument is the same as above.
\end{proof}

\bigskip

For $\text{SO}(n+2)/ \text{SO}(n)\text{SO}(2)$, we compute directly.

If $n=1$, it is a rank one symmetric space. In this case, $\text{SO}(3)/ \text{SO}(2)$ is the sphere.
Through the Borel embedding \textbf{\textit{b}}, its noncompact dual $\text{SO}_{0}(1,2)/ \text{SO}(2)$ in it consists of
those points with distance less than $\frac{1}{4}$ of the diameter of the sphere
from the south pole (directly from Remark 3.3). Hence $\mathrm{Im}(\textbf{\textit{b}})=\exp_o^c(\frac{R}{4})$.
Through the embedding \textbf{\textit{f}}, we have $\mathrm{Im}(\textbf{\textit{f}})=\exp_o^c(\frac{R}{2})$.
Of course, $\textbf{\textit{b}}\neq \textbf{\textit{f}}$.

If $n\geq
2$, it is a rank two symmetric space, follow the notations in
example 4.1, the lattice $\Gamma(G,K)$ is generated by
$\pi(R_{1,n+1}+R_{2,n+2})$ and $\pi(R_{1,n+1}-R_{2,n+2})$ which are
orthogonal to each other and have the same length. By direct computations, we have
$\textbf{\textit{b}}\neq \textbf{\textit{f}}$, but $\mathrm{Im}(\textbf{\textit{b}})=\mathrm{Im}(\textbf{\textit{f}})=\exp_o^c(\frac{R}{2})$.

\bigskip

\subsection{Comparing \textbf{\textit{f}} with \textbf{\textit{g}}}
Finally, we compare the embedding \textbf{\textit{f}} with the generalized embedding \textbf{\textit{g}},
they both coincide with the space-like embedding \textbf{\textit{p}} in the generalized
Grassmannian case  and the Borel embedding \textbf{\textit{b}} in the Hermitian case (except $\text{SO}(n+2)/
\text{SO}(n)\text{SO}(2)$),
when are they equal to each other?

\begin{thm}
For all the symmetric R-spaces except $\text{SO}(n+2)/
\text{SO}(n)\text{SO}(2)$ and $\text{SO}(n+1)\times \text{SO}(m+1)/
S(\text{O}(n) \times \text{O}(m))$, \textbf{\textit{g}}=\textbf{\textit{f}}.
\end{thm}
\begin{proof}
For any symmetric R-space $N$, it is a compact Hermitian symmetric space or it is a real form of a compact Hermitian
symmetric space $M$.

When it is a compact Hermitian symmetric space except $\text{SO}(n+2)/
\text{SO}(n)\text{SO}(2)$, we have \textbf{\textit{g}}=\textbf{\textit{b}}=\textbf{\textit{f}}.

When it is a real form of a compact Hermitian
symmetric space $M$, that means it is a reflective submanifold of $M$ under a antiholomorphic isometric involution.
By Proposition 3.1 and Proposition 4.4, we have $\textbf{\textit{g}}_N=\textbf{\textit{g}}_M|_N$ and $\textbf{\textit{f}}_N=\textbf{\textit{f}}_M|_N$.
If $\textbf{\textit{g}}_M=\textbf{\textit{f}}_M$, then we have $\textbf{\textit{g}}_N=\textbf{\textit{f}}_N$.
From the list of symmetric R-spaces, $\text{SO}(n+m+2)/
\text{SO}(n+m)\text{SO}(2)$ has real forms $\text{SO}(n+1)\times \text{SO}(m+1)/
S(\text{O}(n) \times \text{O}(m))$. Hence $\text{SO}(n+1)\times \text{SO}(m+1)/
S(\text{O}(n) \times \text{O}(m))$ is the only exception.
\end{proof}

Combining all the above results, we have

\begin{prop}
For $\text{SO}(3)/ \text{SO}(2)$,  $\mathrm{Im}(\textbf{\textit{f}})=\exp_o^c(\frac{R}{2})$ and $\mathrm{Im}(\textbf{\textit{g}})=\exp_o^c(\frac{R}{4})$.

For all the symmetric R-spaces except $\text{SO}(3)/ \text{SO}(2)$,
$\mathrm{Im}(\textbf{\textit{f}})=\mathrm{Im}(\textbf{\textit{g}})=\exp_o^c(\frac{R}{2})$.
\end{prop}

Since both the constructions of \textbf{\textit{g}} and \textbf{\textit{f}} use the intrinsic characterizations of symmetric R-spaces among compact symmetric spaces,
they can't be generalized to all compact symmetric spaces.

\bigskip

School of Science, East China University of Science and Technology, Meilong
Road 130, Shanghai, China

E-mail address: yxchen76@ecust.edu.cn

\bigskip

The Department of Mathematics, Jinan University, Guangzhou, Guangdong, China

E-mail address: hyd74@qq.com

\bigskip

The Institute of Mathematical Sciences and Department of Mathematics, The
Chinese University of Hong Kong, Shatin, N.T., Hong Kong

E-mail address: leung@math.cuhk.edu.hk


\bigskip





\begin{thebibliography}{99}
\addcontentsline{toc}{chapter}{Bibliography}


\bibitem{B}
W. Bertram, {\it Ramanujan's master theorem and duality of symmetric spaces},
J. Funct. Anal. 148, no. 1, 117-151, 1997.

\bibitem{BD}
T. $Br\ddot{o}cker$, T. Dieck, {\it Representations of Compact Lie
Groups}, Springer-Verlag, New York Berlin Heidelberg Tokyo, 1985.

\bibitem{H}
S. Helgason, {\it Differential geometry, Lie groups, and Symmetric
spaces}, Academic Press, New York, San Francisco, London, 1978.

\bibitem{HL}
Y.D. Huang and N.C. Leung, {\it A uniform description of compact symmetric spaces as
Grassmannians using the magic square}, Math. Ann. 350, no.1, 79-106, 2011.


\bibitem{K}
A.W. Knapp, {\it Lie Groups Beyond an Introduction, Second Edition},
$Birkh\ddot{a}user$, 2005.

\bibitem{KN}
S. Kobayashi and T. Nagano, {\it On filtered Lie algebras and geometric structures I}, J. Math. Mech. 13, 875-907, 1964.

\bibitem{Ko}
A. Kollross, {\it Duality of symmetric spaces and polar actions},
 J. Lie Theory 21, no. 4, 961-986, 2011.

\bibitem{LNC1}
S.C. Lau and N.C. Leung, {\it Conformal geometry and special holonomy},
Recent advances in geometric analysis, Adv. Lect. Math. (ALM), 11, 195-209, 2010.

\bibitem{LNC2}
N.C. Leung, {\it Riemannian geometry over different normed division algebras},
 J. Differential Geom. 61, no. 2, 289-333, 2002.


\bibitem{LNC3}
N.C. Leung, {\it Geometric structures in Riemannian manifolds},
Handbook of geometric analysis, No. 3, Adv. Lect. Math. (ALM), 14, 129-229,  2010.

\bibitem{L1}
S.P. Leung, {\it The reflection principle for minimal submanifolds
of Riemannian symmetric spaces}, J.Differential geometry. 8,
153-160, 1973.

\bibitem{L2}
S.P. Leung, {\it On the classification of reflective submanifolds of
Riemannian symmetric spaces}, Indiana University Mathematics
Journal, Vol. 24, No.4, 327-339, 1974.

\bibitem{L3}
S.P. Leung, {\it Reflective submanifolds. III. congruency of
isometric reflective submanifolds and corrigenda to the
classification of reflective submanifolds}, J.Differential geometry
14, 167-177, 1979.

\bibitem{L4}
S.P. Leung, {\it Reflective submanifolds. IV. classification of real
forms of Hermitian symmetric spaces}, J.Differential geometry 14,
179-185, 1979.


\bibitem{Lo}
O. Loos, {\it Charakterisierung symmetrischer R-R$\ddot{a}$ume durch ihre Einheitsgitter}, Math. Z. 189, 211-226, 1985.


\bibitem{Mok}
N. Mok, {\it Metric rigidity theorems on Hermitian locally symmetric
manifolds}, Series in Pure mathematics, Volume 6, World
Scientific, Singapore $\cdot$ New Jersey $\cdot$ London $\cdot$
HongKong, 1989.

\bibitem{Na}
T. Nagano, {\it Transformation groups on compact symmetric spaces}, Trans. Amer. Math. Soc. 118, 428-453, 1965.

\bibitem{HN}
H. Naitoh, {\it On cut loci and first conjugate loci of the irreducible symmetric R-space and the
irreducible compact hermitian symmetric spaces}, Hokkaido Mathematical Journal, Vol. 6, 230-242, 1977.

\bibitem{QT}
P. Quast and M.S. Tanaka, {\it Convexity of reflective submanifolds in symmetric R-spaces},  Tohoku Math. J. 64, 607-616, 2012.


\bibitem{TS1}
T. Sakai, {\it On cut loci of compact symmetric spaces}, Hokkaido
mathematical Journal,  Volume 6, 136-161, 1977.

\bibitem{TS2}
T. Sakai, {\it The manifold of the Lagrangean subspaces of a
symplectic vector space}, J.Differential geometry, Volume 12, 555-564,
1977.

\bibitem{TS3}
T. Sakai, {\it On the structure of cut loci in compact Riemannian
symmetric spaces}, Math. Ann, Volume 235, 129-148, 1978.

\bibitem{Ta}
M. Takeuchi, {\it Stability of certain minimal submanifolds of compact Hermitian symmetric spaces}, Tohoku Math. J. (2) 36, 293-314, 1984.

\bibitem{Ta2}
M.S. Tanaka, {\it  Geometry of symmetric R-spaces}, Geometry and analysis on manifolds, Progr. Math. 308,  471-481, 2015.

\bibitem{W}
J.A. Wolf, {\it Fine structure of Hermitain symmetric spaces}, Pure
and App. Math, Volume 8, Dekker New York, 1972.


\bibitem{Y1}
L.Yang, {\it Injectivity radius and Cartan polyhedron for simply connected symmetric spaces},  Chin. Ann. Math. Ser. B 28, no. 6, 685-700, 2007.

\bibitem{Y2}
L.Yang, {\it Injectivity radius for non-simply connected symmetric spaces via Cartan polyhedron}, Osaka J. Math. 45, no. 2, 511-540, 2008.



\end{thebibliography}
\end{document}